\documentclass[a4paper,11pt]{article}
	\usepackage[utf8]{inputenc}
	\usepackage{hyperref}
	\usepackage{amssymb}
	\usepackage{amsmath}
	\usepackage{amsfonts}
	\usepackage{amsthm}
	\usepackage{stmaryrd}
	\usepackage{algorithm,algorithmic,refcount}
	\usepackage{xcolor}
	\usepackage{tkz-euclide}

	\usepackage{epstopdf}
	\usepackage{mathdots}
	\usepackage[normalem]{ulem}
	
	\usepackage{tikz}
	\usetikzlibrary{matrix,decorations.pathreplacing,calc}

	\usepackage{natbib}
	\setcitestyle{square}
	\setcitestyle{numbers}
	\setcitestyle{comma}
	
	\newcommand{\R}{\mathbb {R}}
	
	\newtheorem{theorem}{Theorem}
	\newtheorem{example}[theorem]{Example}
	\newtheorem{corollary}[theorem]{Corollary}
	\newtheorem{lemma}[theorem]{Lemma}
	\newtheorem{remark}[theorem]{Remark}
	\newtheorem{definition}[theorem]{Definition}
	\newcommand{\traceG}{\ensuremath{\,\mathrm{tr}^{\mathsf{G}}}}
	\newcommand{\traceGR}{\ensuremath{\,\mathrm{tr}^{\mathsf{G,R}}}}
	\newcommand{\traceR}{\ensuremath{\,\mathrm{tr}^{\mathsf{R}}}}
	\newcommand{\estG}{\ensuremath{\,\mathrm{est}^{\mathsf{G}}}}
	\newcommand{\estGR}{\ensuremath{\,\mathrm{est}^{\mathsf{G,R}}}}
	\newcommand{\estR}{\ensuremath{\,\mathrm{est}^{\mathsf{R}}}}
    \newcommand{\trace}{\ensuremath{\,\mathrm{tr}}}
    \newcommand{\diag}{\ensuremath{\,\mathrm{diag}}}
    
    \newcommand{\ent}{\ensuremath{\,\mathbb{H}}}
    \newcommand{\D}{\ensuremath{\,\mathrm{D}}}
    \newcommand{\arcsinh}{\ensuremath{\,\mathrm{arcsinh}}}

    \newcommand\new[1]{{{\color{black}#1}}}
    \newcommand\old[1]{} 

    \usepackage{amssymb,amsmath,url}
    \author{Alice Cortinovis\footnote{Institute of Mathematics, EPF Lausanne, 1015 Lausanne, Switzerland. E-mail: alice.cortinovis@epfl.ch. The work of Alice Cortinovis has been supported by the SNSF research project \emph{Fast algorithms from low-rank updates}, grant number: 200020\_178806.} \and Daniel Kressner\footnote{Institute of Mathematics, EPF Lausanne, 1015 Lausanne, Switzerland. E-mail: daniel.kressner@epfl.ch.} }
    \date{}
	  
	\title{On randomized trace estimates for indefinite matrices \\with an application to determinants}
	
\begin{document}
	
	\maketitle

\begin{abstract}

Randomized trace estimation is a popular and well studied technique that approximates the trace of a large-scale matrix $B$ by computing the average of $x^T Bx$ for many samples of a random vector $X$. Often, $B$ is symmetric positive definite (SPD) but a number of applications give rise to indefinite $B$. Most notably, this is the case for log-determinant estimation, a task that
features prominently in statistical learning, for instance in maximum likelihood estimation for Gaussian process regression. The analysis of randomized trace estimates, including tail bounds, has mostly focused on the SPD case. \old{A notable exception is recent work by Ubaru, Chen, and Saad~\cite{Ubaru2017}
on trace estimates for a matrix function $f(A)$ with Rademacher random vectors.} In this work, we derive new tail bounds for randomized trace estimates applied to indefinite $B$ with Rademacher or Gaussian random vectors. These bounds significantly improve existing results for indefinite $B$, reducing 
the number of required samples by a factor $n$ or even more, where $n$ is the size of $B$. Even for an SPD matrix, our work improves an existing result by Roosta-Khorasani and Ascher~\cite{Ascher2015} 
for Rademacher vectors.

This work also analyzes the combination of randomized trace estimates with the  Lanczos method for approximating the trace of $f(A)$. Particular attention is paid to the matrix logarithm, which is needed for log-determinant estimation. We improve and extend an existing result, to not only cover Rademacher but also  Gaussian random vectors.
\end{abstract}




    \section{Introduction}

This paper is concerned with approximating the trace of a symmetric matrix $B \in \R^{n \times n}$ that is accessible only implicitly via matrix-vector products or, more precisely, (approximate) quadratic forms. If $X$ is a random vector of length $n$ such that $\mathbb{E}[X] = 0$ and $\mathbb{E}[XX^T] = I$, then $\mathbb{E}[X^T B X] = \trace(B)$. Based on this result, a stochastic trace estimator~\cite{Hutchinson1989} is obtained from sampling an average of $N$ quadratic forms:
\begin{equation}\label{eq:estN}
     \trace_N(B) := \frac{1}{N} \sum_{i=1}^N (X^{(i)})^T B X^{(i)},
\end{equation}
where $X^{(i)}$, $i = 1,\ldots, N$, are independent copies of $X$.
The most common choices for $X$ are standard Gaussian and Rademacher random vectors. The latter are defined by having i.i.d. entries that take values $\pm 1$ with equal probability. We will consider both choices in this paper and denote the resulting trace estimates by $\traceG_N(B)$ and $\traceR_N(B)$, respectively.

Hutchinson~\cite{Hutchinson1989} used $\traceR_N(B)$ to approximate the trace of the influence matrix of Laplacian smoothing splines. 
In this setting, $B = A^{-1}$ for a symmetric positive definite (SPD) matrix $A$ and, in turn, $A$ is SPD as well.
Other applications, such as spectral density estimation~\cite{Lin2016}, \new{triangle counting in graphs~\cite{Avron2010,Eden2017}}, and determinant computation~\cite{Bai1996}, may feature a symmetric but indefinite matrix $B$.
For approximating the determinant, one exploits the relation 
\begin{equation} \label{eq:logdet}
    \log (\det (A)) = \trace(\log(A)),
\end{equation}
where $\log(A)$ denotes the matrix logarithm of $A$. The need for estimating determinants arises, for instance, in statistical learning~\cite{Affandi2014,Fitzsimons2017,Gardner2018}, lattice quantum chromodynamics~\cite{Thron1998}, and Markov random fields models~\cite{Wainwright2006}. Certain quantities associated with graphs can be formulated as determinants, such as the number of spanning trees, and various negative approximation results exist in this context; see, e.g.,~\cite{Durfee2020,PengWang2018}. Relying on the Cholesky factorization, the exact computation of the determinant is often infeasible for a large matrix $A$. In contrast, the Hutchinson estimator combined with~\eqref{eq:logdet} bypasses the need for factorizing $A$ and instead requires to (approximately) evaluate the quadratic form $x^T \log(A) x$ for several vectors $x \in \R^n$. Compared to the task of estimating the trace of $A^{-1}$, the determinant computation via~\eqref{eq:logdet} is complicated by two issues:  (a) Even when $A$ is SPD, the matrix $B = \log(A)$ may be indefinite; and (b)  the quadratic forms $x^T \log(A) x$ themselves are expensive to compute exactly, so they need to be approximated. 
We mention in passing that there are other methods to approximate \new{traces and} determinants, including \new{randomized} subspace iteration~\cite{Saibaba2017} and block Krylov methods~\cite{LiZhu2020}, but they only work well in specific cases, e.g., when $A = \sigma I + C$ for a matrix $C$ of low numerical rank. \new{The Hutch++ trace estimator, recently proposed and analyzed for the SPD case in~\cite{meyer2021hutch++}, overcomes this limitation via a combination with stochastic trace estimation. Although it is not difficult to imagine that the results presented in this work are useful in extending the analysis from~\cite{meyer2021hutch++} to the indefinite case, a thorough discussion of this extension is beyond the scope of this work. Another direction of work on large-scale determinant estimation has explored the use of
spectral sparsifiers for symmetric diagonally dominant matrices~\cite{Durfee2020,Hunter2014}.}

\paragraph{\new{Trace estimation of indefinite matrices.}}
By the central limit theorem, the estimate~\eqref{eq:estN} can be expected to become more reliable as $N$ increases; see, e.g.,~\cite[Corollaries 3.3 and 4.3]{Chen2016} for such an asymptotic result as $N \to \infty$. Most existing non-asymptotic results for trace estimation are specific to an SPD matrix $B$; see~\cite{AvronToledo2011, Gratton2018, Ascher2015} for examples. They provide a bound on the estimated number $N$ of probe vectors to ensure a small relative error with high probability:
\begin{equation}\label{eq:relSPD}
    \mathbb{P} \left ( \left \lvert \frac{\trace(B) - \traceGR_N(B)}{\trace(B)} \right \rvert \ge \varepsilon \right ) \le \delta;
\end{equation}
see Remark~\ref{rmk:SPDgauss} below for a specific example.
As already mentioned, the assumption that $B$ is SPD is usually not met when computing the determinant of an SPD matrix $A$ via $\trace(\log(A))$ because this would require all eigenvalues of $A$ to be larger than one. For general indefinite $B$, it is unrealistic to aim at a bound of the form~\eqref{eq:relSPD} for the \emph{relative} error, because $\trace(B) = 0$ does not imply zero error. 
Ubaru, Chen, and Saad~\cite{Ubaru2017} derive a bound for the absolute error via rescaling, that is, the results from~\cite{Ascher2015} are applied to the matrix $C := -\log(\lambda A)$ for a value of $\lambda > 0$ that ensures $C$ to be SPD. 
Specifically, for Rademacher vectors it is shown in~\cite[Corollary 4.5]{Ubaru2017} that \begin{equation} \label{eq:proberrorbound}
 \mathbb{P}\left ( | \traceR_N(\log(A)) - \log\det(A) | \ge \varepsilon \right ) \le \delta
\end{equation}
is satisfied with fixed failure probability $\delta$
if the number of samples $N$ \new{grows proportionally to $\varepsilon^{-2} n^2 \log(1 + \kappa(A))^2 \log\frac{2}{\delta}$ where $\kappa(A)$ denotes the condition number of $A$.} 
Unfortunately, this \emph{estimated} number of samples compares unfavorably with a much simpler approach; computing  the trace from the diagonal elements of $\log (A)$ only requires the evaluation of $n$ quadratic forms, using all $n$ unit vectors of length $n$. 
\new{A more general result for indefinite matrices is shown in~\cite{Avron2010} and it also features a worst-case dependence on $n^2$; we refer to 
Remarks~\ref{rmk:AvronG} and~\ref{rmk:AvronR} for a comparison with our new results.}

\paragraph{\new{Approximation of quadratic forms.}}
To approximate the quadratic forms $\new{x^T B x} = x^T \log(A) x$, a polynomial approximation of the logarithm can be used, see~\cite{Han2017,Pace2004} for approximation by Chebyshev expansion/interpolation and~\cite{Barry1999,Boutsidis2017,Zhang2007} for approximation by Taylor series expansion. Often, a better approximation can be obtained by the Lanczos method, which is equivalent to applying Gaussian quadrature to the integral $\int \log(\new{\lambda}) d\mu(\new{\lambda})$ on the spectral interval of $A$, for a suitably defined measure $\mu$; see~\cite{GolubMeurant2010}. In this case, upper and lower bounds for the quantity $x^T \log(A) x$ can be determined without much additional effort~\cite{Bai1996}. Moreover, the convergence of Gaussian quadrature for the quadratic form can be related to the best polynomial approximation of the logarithm on the spectral interval of $A$; see~\cite[Theorem 4.2]{Ubaru2017}.
By combining the polynomial approximation error with~\eqref{eq:proberrorbound}, one obtains a total error bound that takes into account both sources of errors. Such a result is presented in~\cite[Corollary 4.5]{Ubaru2017} for Rademacher vectors; the fact that all such vectors have bounded norm is essential in the analysis. 

\paragraph{\new{Contributions.}}
In this paper, we improve the results from~\new{\cite{Avron2010,Ubaru2017}} by first showing that the number of samples required to achieve~\eqref{eq:proberrorbound} is much lower. In particular, we show for a general symmetric matrix $B$ that 
\begin{equation} \label{eq:mainresult}
    \mathbb{P}\big( | \traceGR_N(B) - \trace(B) | \ge \varepsilon \big) \le \delta
\end{equation}
is satisfied with fixed failure probability $\delta$ if the number of samples $N$ grows proportionally with the stable rank $\rho(B) := \| B \|_F^2 / \| B \|_2^2$; as $\rho(B) \in [1, n]$, the growth is at most linear in $n$ (instead of quadratic). 
We derive such a result for both, Gaussian and Rademacher vectors,
and demonstrate that the dependence on $n$ is asymptotically tight with an explicit example. For SPD matrices \new{$B$}, our 
bound also improves the state-of-the-art result~\cite[Theorem 1]{Ascher2015} for Rademacher vectors by establishing that the number of probe vectors is inversely proportional to the stable rank \new{of $B^{1/2}$}.

Specialized to determinant computation, we combine our results with an improved analysis of the Lanczos method, to get a sharper total error bound for Rademacher vectors. Finally, we extend this combined error bound to Gaussian vectors, which requires some additional consideration because of the unboundedness of such vectors. We remark that some of our results are potentially of wider interest, beyond stochastic trace and determinant estimation, such as a tail bound for Rademacher chaos (Theorem~\ref{thm:concRademacher}) and an error bound (Corollary~\ref{cor:convgeneral} combined with Corollary~\ref{cor:logellipse})    on the polynomial approximation of the logarithm.

\new{We note in passing that some results of this paper also apply to a non-symmetric matrix $B$, because of the relations $\trace(B) = \trace(B_s)$ and $x^T B x = x^T B_{s} x$ with the symmetric part $B_s = (B+B^T)/2$.}

\section{Tail bounds for trace estimates}~\label{sec:absolute}

In this section we derive tail bounds of the form~\eqref{eq:mainresult} for the stochastic trace estimator applied to a symmetric, possibly indefinite matrix $B \in \R^{n\times n}$. 
We will analyze Gaussian and Rademacher vectors separately.
In the following, we will frequently use a spectral decomposition $B = Q\Lambda Q^T$, where $\Lambda = \diag(\lambda_1, \ldots, \lambda_n)$ contains the eigenvalues of $B$ and $Q$ is an orthogonal matrix.
    
\subsection{Standard Gaussian random vectors}    
    
The case of Gaussian vectors will be addressed by using a tail bound for sub-Gamma random variables, which follows from Chernoff bounds; see, e. g., ~\cite{Boucheron2013}.

\begin{definition}
        A random variable $X$ is called \emph{sub-Gamma} with variance parameter $\nu > 0$ and scale parameter $c > 0$ if
        \begin{equation*}
         \mathbb{E}[\exp(\lambda X)] \le \exp \left ( \frac{\nu \lambda^2}{2(1-c\lambda)} \right ) \qquad\text{for all } 0 < \lambda < \frac{1}{c}.
        \end{equation*}
\end{definition}

    \begin{lemma}[{\cite[Section 2.4]{Boucheron2013}}]\label{lemma:chernoff2}
	Let $X$ be a sub-Gamma random variable with parameters $(\nu, c)$. Then, for all $\varepsilon \ge 0$, we have
	\begin{equation*}
	    \mathbb{P}(X \ge\sqrt{2\varepsilon\nu} + c\varepsilon) \le \exp(-\varepsilon).
	\end{equation*}
	\end{lemma}

\begin{lemma}[{\cite[Proposition 2.10]{Wainwright2019}}]\label{lemma:chernoff}
Let $X$ be a random variable such that $\mathbb{E}[X] = 0$, and such that both $X$ and $-X$ are sub-Gamma with parameters $(\nu, c)$. Then, for all $\varepsilon \ge 0$, we have
\begin{equation*}
    \mathbb{P} \left ( |X| \ge \varepsilon \right ) \le 2\exp\left ( -\frac{\varepsilon^2}{2(\nu + c\varepsilon)} \right ).
\end{equation*}
\end{lemma}

Lemma~\ref{lemma:chernoff} implies the following result for the tail of a single-sample trace estimate. This result is similar, but not identical, to~\cite[Example 2.12]{Boucheron2013} and~\cite[Lemma 1]{Laurent2000}, which apply to symmetric matrices with zero diagonal and SPD matrices, respectively.    
    \begin{lemma}\label{lemma:gaussianest}
    For a Gaussian vector $X$ of length $n$ we have
    \begin{equation*}
        \mathbb{P} \left ( |X^T B X - \trace(B)| \ge \varepsilon \right ) \le 2\exp\left ( -\frac{\varepsilon^2}{4\|B\|_F^2 + 4\varepsilon\|B\|_2}  \right )
    \end{equation*}
    for all $\varepsilon>0$.
    \end{lemma}
    \begin{proof}    
    We let
    \begin{equation*}
    Y := X^T B X - \trace(B) = X^T Q \Lambda Q^T X - \trace(B) = \sum_{i=1}^n \lambda_i(Z_i^2-1),
    \end{equation*}
    where $Z_i \sim \mathcal{N}(0,1)$ is the $i$th component of the Gaussian vector $Q^T X$. To show that $Y$ is sub-Gamma, we define for $\lambda \in \R$ the function
    \begin{equation*}
	    \psi(\lambda) := \log\mathbb{E} [ \exp(\lambda(Z^2-1))], \quad Z \sim \mathcal{N}(0,1).
    \end{equation*}
    By direct computation, it follows that $\psi(\lambda) = -\lambda - \frac{1}{2} \log(1-2\lambda)$ for $\lambda<\frac{1}{2}$. In particular, this implies $\psi(\lambda) \le \frac{\lambda^2}{1-2\lambda}$ for 
    $0 \le \lambda < \frac{1}{2}$, and $\psi(\lambda) \le \lambda^2 \le \frac{\lambda^2}{1 + c\lambda}$ \new{for $-\frac{1}{c} < \lambda < 0$ for all $c > 0$}. 
	Using the independence of $Z_i$ for different $i$ we obtain
	\begin{align*}
        \log \mathbb{E}[\exp(\lambda Y)] & = \sum_{i=1}^n \log \mathbb{E} [\exp(\lambda \lambda_i(Z_i^2-1))]
         = \sum_{i=1}^n \psi(\lambda \lambda_i) \\
         & \le \sum_{i=1}^n \frac{\lambda_i^2 \lambda^2}{1-2|\lambda_i| \lambda} \le \frac{\| B \|_F^2 \lambda^2}{1-2\| B \|_2 \lambda}
	\end{align*}
	for $0 < \lambda < \frac{1}{2\| B \|_2}$. This shows that $Y$ is sub-Gamma with parameters $$(\nu,c) = (2\| B \|_F^2, 2\| B \|_2).$$ Moreover, $-Y = X^T (-B) X - \trace(-B)$ is also sub-Gamma with the same parameters. Because $\mathbb{E}[Y]=0$,  Lemma~\ref{lemma:chernoff} implies the desired result.
\end{proof}

A diagonal embedding trick turns Lemma~\ref{lemma:gaussianest} into a tail bound for the stochastic trace estimator~\eqref{eq:estN}.
\begin{theorem}\label{thm:gauss}
Let $B\in \R^{n\times n}$ be symmetric. Then
\begin{equation*}
    \mathbb{P}\left ( | \traceG_N(B) - \trace(B) | \ge \varepsilon \right ) \le 2\exp \left ( -\frac{N\varepsilon^2}{4 \| B \|_F^2 + 4 \varepsilon \| B \|_2}\right )
\end{equation*}
for all $\varepsilon > 0$. In particular, for $N \ge \frac{4}{\varepsilon^2}(\| B \|_F^2 + \varepsilon \| B \|_2) \log\frac{2}{\delta}$ it holds that
    $\mathbb{P}(| \traceG_N(B) - \trace(B) | \ge \varepsilon) \le \delta$.
\end{theorem}
\begin{proof}
    We apply Lemma~\ref{lemma:gaussianest} to the matrix 
    \begin{equation}\label{eq:mathcalB}
    \mathcal{B} := \diag\big ( N^{-1} B,\ldots, N^{-1} B\big) \in \R^{Nn\times Nn},
    \end{equation}
    that is, the block diagonal matrix with the $N$ diagonal blocks containing rescaled copies of $B$. In turn, the trace estimate~\eqref{eq:estN} equals $X^T \mathcal{B} X$ for a Gaussian vector $X$ of length $Nn$. Noting that $\| \mathcal{B} \|_F = N^{-1/2} \| B \|_F$ and $\| \mathcal{B} \|_2 = N^{-1} \| B \|_2$, the first part of the corollary follows from Lemma~\ref{lemma:gaussianest}. Setting 
    \begin{equation*}
        \delta := 2\exp \left ( -\frac{\varepsilon^2}{4 \| \mathcal{B} \|_F^2 + 4 \varepsilon \| \mathcal{B} \|_2}\right ) = 2\exp \left ( -\frac{N\varepsilon^2}{4 \| B \|_F^2 + 4 \varepsilon \| B \|_2}\right )
    \end{equation*}
    we obtain
    $
        N = \frac{4}{\varepsilon^2} \left ( \| B \|_F^2 + \varepsilon \| B \|_2 \right ) \log \frac{2}{\delta}.
    $
\end{proof}

\new{
\begin{remark}
\label{rmk:AvronG}The result of Theorem~\ref{thm:gauss} compares favorably with Lemma 4 in~\cite[]{Avron2010}, which shows that $\mathbb{P}(| \traceG_N(B) - \trace(B) | \ge \varepsilon) \le \delta$ for $N \ge \frac{20}{\varepsilon^2} \| B \|_*^2 \log\frac{4}{\delta}$. Because of $\| B \|_F \le \| B \|_* \le \sqrt{n} \| B \|_F$, the bound of Theorem~\ref{thm:gauss} is always better for reasonable values of $\varepsilon$, and it can improve the estimated number of samples $N$ in~\cite{Avron2010} by a factor proportional to $n$.
\end{remark}
}

We recall that the \emph{stable rank} of $B$ is defined as $\rho = \| B \|_F^2 / \| B \|_2^2$  and satisfies $\rho \in [1,n]$. In particular, $\rho(B) = 1$ when $B$ has rank one and $\rho(B) = n$ when all singular values are equal. Intuitively,  $\rho(B)$ tends to be large when $B$ has many singular values not significantly smaller than the largest one. The minimum number of probe vectors required by Theorem~\ref{thm:gauss} depends on the stable rank of $B$ in the following way: 
\[ \frac{4}{\varepsilon^2}(\rho \| B \|_2^2 + \varepsilon \| B \|_2) \log\frac{2}{\delta}  \le \frac{4}{\varepsilon^2}(n\| B \|_2^2 + \varepsilon \| B \|_2) \log\frac{2}{\delta}.
\]
The upper bound indicates that $N$ may need to be chosen proportionally with $n$ to reach a fixed (absolute) accuracy $\varepsilon$ with constant success probability, provided that $\| B \|_2$ remains constant as well. 
The following lemma shows for a simple matrix $B$ that such a linear growth of $N$ can actually not be avoided. 
\begin{lemma}\label{lemma:tightG}
Let $n$ be even and consider the traceless matrix $B = \begin{bmatrix} I_{\frac{n}{2}} & 0 \\ 0 & -I_{\frac{n}{2}} \end{bmatrix}$. Then, for every $\varepsilon > 0$, it holds that
\begin{equation*}
    \mathbb{P}(| \traceG_N(B)| \le \varepsilon) \le \varepsilon \sqrt{\frac{N}{\pi n}}. 
\end{equation*}
\end{lemma}
\begin{proof}
By the definition of $B$, the trace estimate takes the form
\begin{equation*}
    \traceG_N(B) = \frac{1}{N} \Big( \sum_{i=1}^{nN/2} X_i^2 - \sum_{j=1}^{nN/2} Y_j^2 \Big)
\end{equation*}
for independent $X_i, Y_j \sim N(0,1)$. In other words,
\begin{equation*}
    N \cdot \traceG_N(B) = X - Y, 
\end{equation*}
where $X, Y$ 
 are independent 
Chi-squared random variables with $\frac{nN}{2}$ degrees of freedom.
The probability density function $f$ of $Z = X-Y$ can be expressed as
\begin{equation*}
    f(z) = \frac{1}{2^{nN/4} \sqrt{\pi} \Gamma(nN/4)} |z|^{\frac{nN}{4}-\frac{1}{2}} K_{\frac{nN}{4}-\frac{1}{2}}(|z|),
\end{equation*}
where $K_{\frac{nN}{4}-\frac{1}{2}}$ is a modified Bessel function of the second kind~\cite{diffChiSquared}. 
In particular, 
\begin{equation*}
f(0) = \frac{1}{4\sqrt{\pi}} \frac{\Gamma\left ( \frac{nN}{4} - \frac{1}{2} \right )}{\Gamma \left ( \frac{nN}{4} \right )} = \frac{1}{4\sqrt{\pi}} \frac{\sqrt{\pi} }{2^{\frac{nN}{2}-2}}\binom{\frac{nN}{2}-2}{\frac{nN}{4}-1}\le \frac{1}{2\sqrt{\pi nN}},
\end{equation*}
where we used the duplication formula for Gamma functions and the inequality $\frac{1}{2^{2k}}\binom{2k}{k} \le \frac{1}{\sqrt{\pi k}}$; see~\cite{centralBinCoeff}.

As $f$ is an autocorrelation function (of the density function of a Chi-squared variable with $nN/2$ degrees of freedom), its maximum is at $0$. We can therefore estimate the probability of $X-Y$ being in the interval $[-N\varepsilon, N\varepsilon]$ in the following way:
\begin{equation*}
    \mathbb{P}(|\traceG_N(B)| \le \varepsilon) = \mathbb{P}(|X-Y| \le N\varepsilon) \le 2N\varepsilon f(0) \le \varepsilon\sqrt{\frac{N}{\pi n}}.\qedhere
\end{equation*}
\end{proof}

We can reformulate Theorem~\ref{thm:gauss} in such a way that, given a number $N$ of probe vectors and a failure probability $\delta \in (0,1)$, we have $\varepsilon = \varepsilon(B, N, \delta)$ such that with probability at least $1-\delta$ one has $\traceG_N(B) \in [\trace(B) - \varepsilon, \trace(B) + \varepsilon]$. The random variable $X^T \mathcal{B} X - \trace(\new{\mathcal{B}})$, where $\mathcal{B}$ is defined as in~\eqref{eq:mathcalB} and $X$ is a Gaussian vector of length $nN$, is sub-Gamma with parameters $\left ( 2\frac{\|B\|_F^2}{N}, 2\frac{\|B\|_2}{N} \right )$, and the same holds for $-X^T \mathcal{B} X$. By Lemma~\ref{lemma:chernoff2} we have 
\begin{equation}\label{eq:epsBoundG}
    \varepsilon \equiv \varepsilon(B, \delta, N) = \frac{2}{\sqrt{N}} \| B \|_F \sqrt{\log\frac{2}{\delta}} + \frac{2}{N} \| B \|_2 \log\frac{2}{\delta} \le \Big( 2 \sqrt{\frac{n}{N}\log\frac{2}{\delta}} + \frac{2}{N}  \log\frac{2}{\delta}\Big) \| B \|_2.
\end{equation}
As the example in Lemma~\ref{lemma:tightG} shows, the potential growth of $\varepsilon$ with $\sqrt{n}$ cannot be avoided in general. Figure~\ref{fig:tightG} illustrates this growth. In the case of relative error estimates for symmetric positive semidefinite (SPSD) matrices, it is shown in~\cite{Wimmer2014} that the dependence on $\log\frac{2}{\delta}$ and $\frac{1}{\varepsilon^2}$ cannot be improved.

    \begin{figure}
    \centering
    \includegraphics{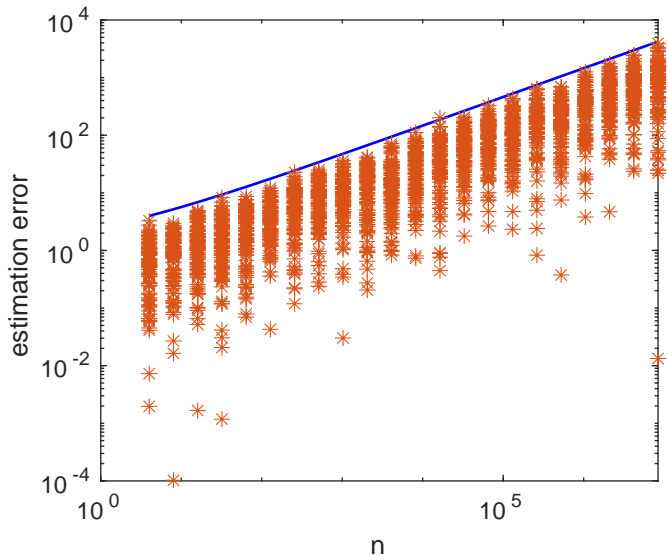}
    \caption{Asterisks: Errors $|\traceG_{10}(B) - \trace(B)|$ of $100$ samples for each $n = 2^k$ with $k = 2, \ldots, 23$, where $B$ is the matrix from 
     Lemma~\ref{lemma:tightG}. Blue line: Error bound $\varepsilon(B,0.01,10)$ from~\eqref{eq:epsBoundG}.}
    \label{fig:tightG}
    \end{figure}

    \old{Also the dependence of $N$ on $\log\frac{2}{\delta}$ and $\frac{1}{\varepsilon^2}$ in Theorem~\ref{thm:gauss} cannot be improved, as shown in (Wimmer et al., 2014)}
    
        \begin{remark}\label{rmk:SPDgauss}
     For a nonzero symmetric positive semidefinite (SPSD) matrix $B$, \new{the result of Theorem~\ref{thm:gauss}  can be turned into a relative error estimate. Let $\mu := \| B \|_2 / \trace(B) = \rho(B^{1/2})^{-1}$.
     Replacing $\varepsilon$ by $\varepsilon \cdot \trace(B)$ in Theorem~\ref{thm:gauss} and noting that
     $\| B \|_F^2 / \trace(B)^2 \le \mu$, one obtains
        \begin{equation*}
            \mathbb{P} \left ( \frac{|\traceG_N(B)-\trace(B)|}{\trace(B)} \ge \varepsilon \right ) \le \delta \quad\text{ for } N \ge \frac{4}{\varepsilon^2}(1 + \varepsilon)\mu \log\frac{2}{\delta}.
        \end{equation*}
        State-of-the-art results of a similar form are Theorem 3 in~\cite{Ascher2015}, which requires $N \ge \frac{8}{\varepsilon^2} \mu \log\frac{2}{\delta}$, and Corollary 3.3 in~\cite{Gratton2018}, which requires $N \ge \frac{2}{\varepsilon^2} \new{\mu} \log\frac{2}{\delta}$ \emph{and}  $\varepsilon \in \left (0, \frac{1}{2}\right )$. 
        Compared to~\cite{Gratton2018}, our result imposes no restriction on $\varepsilon$  at the expense of a somewhat larger constant.
        On the other hand, as $\varepsilon \le 1$, our result is always more favorable than the result from~\cite{Ascher2015} for SPSD matrices.}
    \end{remark}

    \subsection{Rademacher random vectors}
    
    The quadratic form $X^T B X$ for a Rademacher vector $X$ is called \emph{Rademacher chaos of order $2$}. We will first consider the homogeneous case, corresponding to a matrix $B$ with zero diagonal, which has been studied extensively in the literature, see, e.g.,~\cite{Boucheron2013,FoucartRauhut2013,HansonWright1971,Krahmer2011,Talagrand1996}. The non-homogeneous case is easily obtained from the homogeneous case; see Corollary~\ref{cor:rademacher} below. We make use of the the entropy method~\cite{Boucheron2013} to establish the following tail bound for a single-sample trace estimate.

    \begin{theorem}\label{thm:concRademacher}
        Let $X$ be a Rademacher vector of length $n$ and let $B$ be a nonzero symmetric matrix such that $B_{ii} = 0$ for $i=1, \ldots, n$. Then, for all $\varepsilon > 0$,
        \begin{equation}\label{eq:radConfInt}
            \mathbb{P} \left ( | X^T B X | \ge \varepsilon \right ) \le 2\exp\left ( -\frac{\varepsilon^2}{8\|B\|_F^2 + 8\varepsilon \|B\|_2 } \right ).
        \end{equation}
    \end{theorem}

\begin{proof}
The proof follows closely~\cite[Theorem 6]{Adamczak2003} and~\cite[Theorem 17]{Boucheron2013}; see Remark~\ref{rmk:compare} for a comparison with these results. The main idea of the proof is as follows. Using the logarithmic Sobolev inequalities discussed in the appendix, a bound on the entropy of the random variable $X^T B X$ is obtained. Using a (modified) Herbst argument, we derive a bound on the moment generating function (MGF) of $X^T B X$, establishing that it is sub-Gamma with certain constants, which then allows us to apply Lemma~\ref{lemma:chernoff}.

Without loss of generality, we may assume $\| B \|_2 = 1$; the general case follows from applying the result to $\tilde B := B / \|B\|_2$. 
Let us consider the function $f:\{-1,1\}^n \to \R$ defined as
\begin{equation*}
    f(x) = x^T B x = \sum_{i \neq j} x_i x_j B_{ij}.
\end{equation*}
We want to apply the logarithmic Sobolev inequality~\eqref{eq:grossnormal} from Theorem~\ref{thm:logsobolev} to $f(X)$. For this purpose, we let \[
\bar{X}^{(i)} = \big[ X_1, \ldots, X_{i-1}, -X_i, X_{i+1}, \ldots, X_n \big]^T = X - 2 X_i e_i, \quad i=1, \ldots, n,
\]
where $e_i$ denotes the $i$th unit vector. Using that
$B$ has zero diagonal entries, we obtain
\[
f(X) - f(\bar{X}^{(i)})  = \langle B X, X\rangle - \langle BX - 2 X_i B e_i, X - 2 X_i e_i \rangle = 4 X_i \langle B e_i, X  \rangle. 
\]
Therefore, denoting
\[
Y := \|BX\|_2^2 = \frac{1}{16} \sum_{i=1}^n \Big( \sum_{j=1}^n B_{ij}X_j\Big)^2 ,
\]
\new{We denote by $\mathbb{H}(Z)$ the entropy\footnote{\new{The entropy of $Z$ is defined as $\ent(Z) := \mathbb{E}[Z \log Z] - \mathbb{E}[Z] \log \mathbb{E}[Z]$ provided that all expected values exist.}} of a random variable $Z$}. Theorem~\ref{thm:logsobolev} establishes, for all $\lambda > 0$,
\begin{equation}\label{eq:219}
    \ent(\exp(\lambda f(X))) \le 2\lambda^2 \mathbb{E} \left [ Y \exp(\lambda f(X)) \right ].
\end{equation}

The decoupling inequality in~\cite[Lemma 8.50]{FoucartRauhut2013}, which follows from Jensen's inequality, gives
\begin{equation*}
    \lambda \mathbb{E}[Y\exp(\lambda f(X))] \le \ent(\exp(\lambda f(X))) + \mathbb{E}[\exp(\lambda f(X))]\log \mathbb{E}[\exp(\lambda Y)].
\end{equation*}
Combined with~\eqref{eq:219}, this implies
\begin{equation}\label{eq:220}
\ent(\exp(\lambda f(X))) \le \frac{2\lambda}{1-2\lambda} \mathbb{E}[\exp(\lambda f(X))] \cdot \log\mathbb{E}[\exp(\lambda Y )] \text{ for } 0 < \lambda < \frac{1}{2}.
\end{equation}

To find an upper bound on the MGF of $Y$, we use again a logarithmic Sobolev inequality, then transform the obtained bound on the entropy into a bound on the MGF by Herbst argument. We do so by applying the inequality~\eqref{eq:grossplus} from Theorem~\ref{thm:logsobolev} to the function $h: \R^n \to \R$ defined by $h(x) := \| B x \|_2^2$. For this purpose, note that
\begin{eqnarray*}
h(X) - h(\bar{X}^{(i)}) &=& \langle B X, BX \rangle -  
\langle B \bar{X}^{(i)}, B \bar{X}^{(i)} \rangle 
= \langle B (X- \bar{X}^{(i)}), B (X+ \bar{X}^{(i)}) \rangle \\
&=& 4 \langle X_i B e_i, BX - X_i B e_i \rangle \le  
4 X_i \langle  B e_i, BX \rangle
\end{eqnarray*}
and, hence,
\[
\sum_{i=1}^n \left (h(X) - h(\bar{X}^{(i)})\right )_+^2 \le 
16 \sum_{i=1}^n \langle  B e_i, BX \rangle^2
= 16 \|B^T BX\|_2^2 \le
16 \|BX\|_2^2.
\]
Therefore Theorem~\ref{thm:logsobolev} gives
\begin{equation*}
    \ent(\exp(\lambda Y)) \le 4 \lambda^2 \mathbb{E} [Y \exp( \lambda Y)].
\end{equation*}
Letting $g(\lambda) := 4\mathbb{E}[Y \exp(\lambda Y)] /\mathbb{E}[ \exp(\lambda Y)]$, we have obtained a bound of the form~\eqref{eq:ineqH}, as required by Lemma~\ref{lemma:herbst}. Note that $g(\lambda) = 4 \psi'(\lambda)$, where $\psi(\lambda) := \log\mathbb{E}[\exp(\lambda Y)]$. The result of Lemma~\ref{lemma:herbst} gives
\begin{equation*}
    \log \mathbb{E}[\exp(\lambda Y)] \le \frac{\lambda}{1-4\lambda} \| B \|_F^2 \,\,\text{ for }\lambda\in\left (0, \frac{1}{4}\right ).
\end{equation*}
Inserting this inequality into~\eqref{eq:220} gives
\begin{equation*}
    \ent(\exp(\lambda f(X))) \le \frac{2\lambda^2 \| B \|_F^2}{(1-4\lambda)(1-2\lambda)} \mathbb{E}[\exp(\lambda f(X))] \,\,\text{ for }\lambda\in\left (0, \frac{1}{4}\right ).
\end{equation*}
The random variable $f(X)$ satisfies~\eqref{eq:ineqH} for the function $g(\lambda) := \frac{2\|B\|_F^2}{(1-4\lambda)(1-2\lambda)}$ in the interval $[0, 1/4)$. Recalling that $\mathbb{E}[f(X)]=0$, the result of Lemma~\ref{lemma:herbst} gives
\begin{equation*}
    \log \mathbb{E}[\exp(\lambda f(X))] \le \lambda \| B \|_F^2 \log \frac{1-2\lambda}{1-4\lambda} \le \frac{2\| B \|_F^2 \lambda^2}{1-4\lambda},
    \quad \lambda \in [0, 1/4),
\end{equation*}
where we used $\log(1+x) \le x$ in the last inequality. 

Replacing $f$ by $-f$ and $B$ by $-B$, we also obtain
\begin{equation*}
    \log \mathbb{E}[\exp(- \lambda f(X))] \le \frac{2\| B \|_F^2 \lambda^2}{1-4\lambda}, \quad \lambda \in [0, 1/4).
\end{equation*}
Therefore both variables $f(X)$ and $-f(X)$ are sub-Gamma with parameters $( 4\|B\|_F^2, 4)$. Applying Lemma~\ref{lemma:chernoff} concludes the proof.
\end{proof}
\begin{remark}\label{rmk:compare}
The proof of Theorem~\ref{thm:concRademacher} follows the proof of~\cite[Theorem 6]{Adamczak2003}, which in turn refines a result from~\cite[Theorem 17]{Boucheron2003} (see also~\cite{Boucheron2013}) by substituting the more general logarithmic Sobolev inequality from~\cite[Proposition 10]{Boucheron2003} with the ones from  Theorem~\ref{thm:logsobolev} specific for Rademacher random variables. However, let us stress that the results in~\cite{Adamczak2003,Boucheron2003} feature larger constants partly because they deal with the more general Rademacher chaos 
\begin{equation}\label{eq:supchaos}
    f(X) = \sup_{B \in \mathcal{B}} \sum_{i\neq j} X_i X_j B_{ij},
\end{equation}
where $\mathcal{B}$ is a set of symmetric matrices 
with zero diagonal. Restricted to the case $\mathcal{B} = \{B\}$, the results stated in~\cite[Theorem 6]{Adamczak2003} and~\cite[Exercise 6.9]{Boucheron2013} give $\mathbb{P} \left ( | X^T B X | \ge \varepsilon \right ) \le 2\exp\left ( -\frac{\varepsilon^2}{16\|B\|_F^2 + 16 \|B\|_2 \varepsilon} \right )$ and $\mathbb{P} \left ( | X^T B X | \ge \varepsilon \right ) \le 2\exp\left ( -\frac{\varepsilon^2}{32\|B\|_F^2 + 128 \|B\|_2 \varepsilon} \right )$, respectively.\\ Proposition 8.13 in~\cite{FoucartRauhut2013} \old{, which also aims at the more general~\eqref{eq:supchaos},} states $\mathbb{P} \left ( | X^T B X | \ge \varepsilon \right ) \le 2\exp\left ( -\min\left \{ \frac{3\varepsilon^2}{128 \| \new{B} \|_F^2}, \frac{\varepsilon}{32 \| \new{B} \|_2}  \right \} \right ) $.
\end{remark}
    
    As for Gaussian vectors, the result of Theorem~\ref{thm:concRademacher} can be turned into a tail bound for $\traceR_N(B)$ by block diagonal embedding. In the following, let $\D_B$ denote the diagonal matrix containing the diagonal entries of $B$.
        \begin{corollary}\label{cor:rademacher}
        Let $B$ be a nonzero symmetric matrix. Then 
        \begin{equation*}
             \mathbb{P}\big( | \traceR_N(B) - \trace(B) | \ge \varepsilon \big) \le 2\exp \left ( -\frac{N\varepsilon^2}{8 \| B - \D_B \|_F^2 + 8 \varepsilon \| B - \D_B \|_2}\right )
        \end{equation*}
        for every $\varepsilon > 0$. In particular, for 
     \begin{equation*}
     N \ge \frac{8}{\varepsilon^2} \left ( \| B - \D_B \|_F^2 + \varepsilon \| B - \D_B \|_2 \right ) \log \frac{2}{\delta} 
     \end{equation*}
     it holds that
    $\mathbb{P}\big(| \traceR_N(B) - \trace(B) | \ge \varepsilon \big) \le \delta$.
    \end{corollary}
    \begin{proof}
    Let $C := B - \D_B$ and $\mathcal{C} := \diag\big( N^{-1} C,\ldots, N^{-1} C\big) \in \R^{Nn\times Nn}$. Then, $\traceR_N(B) - \trace(B) = X^T \mathcal{C} X$ for a Rademacher vector $X$ of length $Nn$.
    The matrix $\mathcal{C}$ has zero diagonal, $\| \mathcal{C} \|_F = N^{-1/2} \new{\| C \|_F}$, and $\| \mathcal{C} \|_2 = N^{-1} \|C\|_2$. Now, the first part of the corollary directly follows from Theorem~\ref{thm:concRademacher}. Imposing a failure probability of $\delta$ in~\eqref{eq:radConfInt} gives
    \begin{equation*}
        \delta := 2\exp \left ( -\frac{\varepsilon^2}{8 \| \mathcal{C} \|_F^2 + 8 \varepsilon \| \mathcal{C} \|_2}\right ) = 2\exp \left ( -\frac{N\varepsilon^2}{8 \| C \|_F^2 + 8 \varepsilon \| C \|_2}\right ),
    \end{equation*}
    and hence
    $
        N = \frac{8}{\varepsilon^2} \left ( \| C \|_F^2 + \varepsilon \| C \|_2 \right ) \log \frac{2}{\delta}    $. \qedhere
    \end{proof}

\begin{remark}
It is instructive to compare the result of Corollary~\ref{cor:rademacher} to the straightforward application of Bernstein's inequality, which gives
\begin{equation*}
    \mathbb{P}\left ( | \traceR_N(B) - \trace(B) | \ge \varepsilon \right ) \le 2\exp \left ( -\frac{N\varepsilon^2}{4 \| B - \D_B \|_F^2 + \frac{4}{3}  n \varepsilon \| B - \D_B \|_2}\right ).
\end{equation*}
Clearly, a disadvantage of this bound is the explicit dependence of the denominator on $n$, which does not appear in Corollary~\ref{cor:rademacher}.
\end{remark}

An alternative expression for the lower bound on $N$ is obtained by noting that $\| B - \D_B\|_F \le \|B\|_F$ and $\| B - \D_B\|_2 \le 2 \| B \|_2$ (the factor $2$ in the latter inequality is asymptotically tight, see, e.g.,~\cite{Bhatia1989}). The result of  Corollary~\ref{cor:rademacher} thus states that $N$ needs to be at least in the following way: 
\[ \frac{8}{\varepsilon^2}(\rho \| B \|_2^2 + 2 \varepsilon \| B \|_2) \log\frac{2}{\delta}  \le \frac{8}{\varepsilon^2}(n \| B \|_2^2 + 2 \varepsilon \| B \|_2) \log\frac{2}{\delta},
\]
where $\rho$ is the stable rank of $B$.

 \new{\begin{remark}\label{rmk:AvronR} In analogy to the Gaussian case (see Remark~\ref{rmk:AvronG}), the result of Corollary~\ref{cor:rademacher} compares favorably with~Lemma 5 in~\cite{Avron2010}, which shows that $\mathbb{P}(| \traceR_N(B) - \trace(B) | \ge \varepsilon) \le \delta$ for $N \ge \frac{6}{\varepsilon^2} \| B \|_*^2 \log\frac{2\cdot \mathrm{rank}(B)}{\delta}$.
    \end{remark}}

In analogy to the Gaussian case, the following lemma shows that a potential linear dependence of $N$ on $n$ cannot be avoided in general.
\begin{lemma}
    Let $n$ be even and consider the traceless matrix $B = \begin{bmatrix} & & & 1 \\ & & 1 & \\ & \iddots & & \\ 1 & & & \end{bmatrix}$. Then
    \begin{equation*}
        \mathbb{P} \left ( |\traceR_N(B)| \le \varepsilon \right ) \le \varepsilon \sqrt{\frac{N}{\pi n}} 
    \end{equation*}
    for every $\varepsilon > 0$.
\end{lemma}
\begin{proof}
We first note that $\traceR_N(B) = \frac{2}{N} \sum_{i=1}^{nN/2} Z_i$ with independent  Rademacher random variables $Z_i$. 
In turn, $ \mathbb{P} \left ( |\traceR_N(B)| \le \varepsilon \right ) = \mathbb{P} \left ( \left \lvert \sum_{i=1}^{nN/2} Z_i \right \rvert \le \frac{N\varepsilon}{2} \right )$ equals the probability that the number of variables satisfying $Z_i = 1$ is at least $\frac{n-\varepsilon}{4}N$ and at most $\frac{n+\varepsilon}{4}N$. Therefore,
\begin{align*}
    \mathbb{P} \big( |\traceR_N(B)| \le \varepsilon \big) & = \frac{1}{2^{nN/2}} \sum_{i=\lceil\frac{n-\varepsilon}{4}N\rceil}^{\lfloor\frac{n + \varepsilon}{4}N\rfloor} \binom{nN/2}{i} \le \frac{N\varepsilon}{2} \cdot \frac{1}{2^{nN/2}}\cdot \binom{nN/2}{nN/4} \\
    & \le \frac{N\varepsilon}{2} \cdot \frac{2}{\sqrt{\pi nN}} = \varepsilon \sqrt{\frac{N}{\pi n}},
\end{align*}
where we used the inequality $\frac{1}{2^{2k}}\binom{2k}{k} \le \frac{1}{\sqrt{\pi k}}$. 
\end{proof}
We do not report a figure analogous to Figure~\ref{fig:tightG} because the observed errors are very similar to the Gaussian case.

For SPSD matrices, a relative error estimate follows from Corollary~\ref{cor:rademacher} similarly to what has been discussed in Remark~\ref{rmk:SPDgauss} for Gaussian vectors. We recall that $\rho = \| B \|_F^2 / \| B \|_2^2$ denotes the stable rank of $B$.
    \begin{corollary} \label{cor:spdrademacher}
        For a nonzero SPSD matrix $B$, we have 
        \begin{equation*}
            \mathbb{P} \left ( \frac{|\traceR_N(B)-\trace(B)|}{\trace(B)} \ge \varepsilon \right ) \le \delta \quad \text{for} \quad N \ge \frac{8}{\varepsilon^2}(1  + \varepsilon)\new{\mu} \log\frac{2}{\delta}, \quad \new{\text{where} \quad \mu := \frac{\|B\|_2}{\trace(B)}}.
        \end{equation*}
    \end{corollary}
    \begin{proof}
     First of all, it is immediate that $\| B - \D_B \|_F \le \| B \|_F$. As shown, e.g., in~\cite[Theorem 4.1]{Bhatia1989}, 
    the same holds for the spectral norm when $B$ is SPSD. For convenience, we provide a short proof: For every $y\in \R^n$ it holds that
     \begin{equation*}
        | y^T (B - \D_B) y| \le \max\{y^T B y, y^T \D_B y\} \le \max\{ \| B \|_2, \| \D_B \|_2 \} \le \| B \|_2,
     \end{equation*}
     where the first inequality uses that both $y^T B y$ and $y^T \D_B y$ are nonnegative. By taking the maximum with respect to all vectors of norm $1$ one obtains $\| B - \D_B \|_2$ on the left-hand side, which shows that it is bounded by $\| B \|_2$.
     \new{Now, the claimed result follows from Corollary~\ref{cor:rademacher} using the arguments of Remark~\ref{rmk:SPDgauss}.}
    \end{proof}
 
Corollary~\ref{cor:spdrademacher} improves the result from~\cite[Theorem 1]{Ascher2015}, which requires $N \ge \frac{6}{\varepsilon^2} \log\frac{2}{\delta}$; a lower bound that does not improve as \new{$\mu$ decreases.} 

    \section{Lanczos method to approximate quadratic forms}\label{sec:Lanczos}
    
    Let us now consider the problem of estimating the log determinant through $\log(\det(A)) = \trace(\log(A))$, or more generally the problem of computing the trace of $f(A)$ for an analytic function $f$.
    
    Applying a trace estimator to $\trace(f(A))$ requires the (approximate) computation of the quadratic forms $x^T f(A) x$ for fixed vectors $x \in \R^n$. Following~\cite{Ubaru2017}, we use the Lanczos method, Algorithm~\ref{alg:lanczos}, for this purpose.
    
    \begin{algorithm}
	\caption{Lanczos method to approximate quadratic form $x^T f(A) x$ \label{alg:lanczos}}
	 \begin{algorithmic}[1]
	 \REQUIRE{Matrix $A \in \R^{n\times n}$, nonzero vector $x\in\R^n$, number of iterations $m$}
	 \ENSURE{Approximation of $x^T f(A) x$}
     \STATE{Initialize $u_1 \leftarrow x / \|x\|_2$ and $\beta_0\leftarrow 0$}
     \FOR{$i = 1, \ldots, m$}
        \STATE{$\alpha_i \leftarrow u_i^T A u_i$}
        \STATE{$r_i \leftarrow Au_i - \alpha_i u_i - \beta_{i-1} u_{i-1}$}
        \STATE{$\beta_i \leftarrow \| r_i \|_2$}
        \STATE{$u_{i+1} \leftarrow r_i/\beta_i$}
    \ENDFOR
    \STATE{$T_m \leftarrow \begin{scriptsize}\begin{bmatrix} \alpha_1 & \beta_1 & & \\ \beta_1 & \alpha_2 & \ddots & \\ & \ddots & \ddots & \beta_{m-1} \\ & & \beta_{m-1} & \alpha_m \end{bmatrix}\end{scriptsize}$}
    \STATE{Return $\|x\|_2^2 \cdot e_1^T f(T_m) e_1$}
	\end{algorithmic}
	\end{algorithm}
    
    For theoretical considerations, it is helpful to view the quadratic form as an integral.
    \new{For this purpose, we consider the spectral decomposition $A = Q\Lambda Q^T$, $\Lambda = \diag(\lambda_1, \ldots, \lambda_n)$, with $\lambda_{\min} = \lambda_1 \le \lambda_2 \le \cdots \le \lambda_n = \lambda_{\max}$. Then 
    \[
    x^T f(A) x = I:= \int_{\lambda_{\min}}^{\lambda_{\max}} f(\lambda) \,\text{d}\mu(\lambda),
    \]
    with the piecewise constant measure
    \begin{equation}\label{eq:mu}
        \mu(\lambda) := \sum_{i=1}^n z_i^2 \chi_{[\lambda_i,\infty)}(\lambda),\quad z := Q^T x,
    \end{equation}
    where $\chi$ denotes the indicator function.}
It is well known~\cite[Theorem 6.2]{GolubMeurant2010} that the approximation $I_m$ returned by the $m$-points Gaussian quadrature rule applied to $I$ is identical to the approximation returned by $m$ steps of the Lanczos method:
    \begin{equation*}
    I_m := \| x \|_2^2 \cdot \new{e_1^T f(T_m) e_1}.
    \end{equation*}    
    
    To bound the error $|I - I_m|$, the analysis in~\cite{Ubaru2017} proceeds by using existing results on the polynomial approximation error of analytic functions. Although our analysis is along the same lines, it differs in a key technical aspect; we derive and use an improved error bound for the approximation of the logarithm; see Corollary~\ref{cor:numLanczosSteps}. We have also noted two minor erratas in~\cite{Ubaru2017}; see the proof of Theorem~\ref{thm:ellipse} and the remark after Corollary~\ref{cor:convgeneral} for details.
    
    \begin{theorem} \label{thm:ellipse}
         Let $f:[-1,1] \to \R$ admit an analytic continuation to a Bernstein ellipse $E_{\rho_0}$ with foci $\pm 1$ and elliptical radius $\rho_0$. For $1 < \rho < \rho_0$, let $M_{\rho}$ be the maximum of $|f(z)|$ on
         $E_{\rho}$.
         Then 
         \begin{equation*}
            | I - I_m | \le \| x \|_2^2 \cdot \frac{4 M_{\rho}}{1-\rho^{-1}} \rho^{-2m}.
         \end{equation*}
    \end{theorem}
\begin{proof} 
As in~\cite{Ubaru2017}, this result follows directly from bounds on the  polynomial approximation error of analytic functions via Chebyshev expansion, combined with the fact that 
$m$-points Gaussian quadrature is exact for polynomials up to degree $2m-1$.
However, the proof of~\cite[Theorem 4.2]{Ubaru2017} uses an extra ingredient, which seems to be wrong. It claims that the integration error for odd-degree Chebyshev polynomials is zero thanks to symmetry. While this fact is indeed true for the standard Lebesgue measure, it does not hold for the measure~\eqref{eq:mu}. In turn, one obtains the slightly worse factor $1-\rho^{-1}$ in the denominator, compared to the factor $1-\rho^{-2}$ that would have been obtained from~\cite[Theorem 4.2]{Ubaru2017} translated into our setting.
\end{proof}

The affine linear transformation
\begin{equation*}
        \varphi:[\lambda_{\min}, \lambda_{\max}] \to [-1, 1], \qquad x \mapsto \frac{2}{\lambda_{\max} - \lambda_{\min}}t - \frac{\lambda_{\max} + \lambda_{\min}}{\lambda_{\max} - \lambda_{\min}},
\end{equation*}
is used to map an interval $[\lambda_{\min}, \lambda_{\max}]$ containing the eigenvalues of $A$ to the interval $[-1,1]$ of Theorem~\ref{thm:ellipse}. Defining $g:= f \circ \varphi^{-1}$, one has
\begin{equation} \label{eq:relationsfg}
x^T g(\varphi(A)) x = x^T f(A) x, \quad e_1^T g(\varphi(T_m))e_1 = e_1^T f(T_m) e_1.
\end{equation}
By its shift \new{and scaling} invariance, the Lanczos method with $g$, $\varphi(A)$, and $x$ returns the approximation $e_1^T g(\varphi(T_m))e_1$. This allows us to  apply Theorem~\ref{thm:ellipse}. Combined with the relations~\eqref{eq:relationsfg}, the following result is obtained. 
\begin{corollary} \label{cor:convgeneral} With the notation introduced above, it holds that
         \begin{equation*}
            \left\lvert x^T f(A) x - \|x\|_2^2 \cdot e_1^T f(T_m) e_1  \right\rvert \le \| x \|_2^2 \cdot \frac{4 M_{\rho}}{1-\rho^{-1}} \rho^{-2m},
         \end{equation*}
\end{corollary}
Note that $M_\rho$ is the maximum of $g$ on $E_\rho$, which is equal to the maximum of $f$ on the transformed ellipse with foci $\lambda_{\min}, \lambda_{\max}$,
and elliptical radius $(\lambda_{\max}-\lambda_{\min}) \rho /2$. 
The result of Corollary~\ref{cor:convgeneral} differs from the corresponding result in~\cite[page 1087]{Ubaru2017}, which features an additional, erroneous factor $(\lambda_{\max}(A) - \lambda_{\min}(A))/2$.
    
    For the special case of the logarithm, the following result is obtained.
    \begin{corollary}\label{cor:numLanczosSteps}
        Let $A \in \R^{n \times n}$ be SPD with condition number $\kappa(A)$,
        $f\equiv \log$ and $x \in \R^n \backslash\{0\}$. Then the error of the Lanczos method after $m$ steps satisfies
        \begin{equation*}
            | x^T \log(A) x -  \|x\|_2^2 \cdot e_1^T \log(T_m) e_1 | \le  \new{c_A} \|x\|_2^2\left( \frac{\sqrt{\kappa(A)+1}-1}{\sqrt{\kappa(A)+1}+1} \right)^{2m}.
        \end{equation*}
        \new{where $c_A := 2 (\sqrt{\kappa(A)+1}+1) \log(2\kappa(A))$.}
    \end{corollary}
    
    \begin{proof}
    The proof consists of applying Corollary~\ref{cor:convgeneral} to a rescaled matrix. More specifically, we choose $B := \lambda A$ 
    with $\lambda := 1/ ( 2\lambda_{\min} ) > 0$. The tridiagonal matrix returned by the Lanczos method with $A$ replaced by $B$ satisfies $T^B_m = \lambda T_m$. Together with the identity $\log(\lambda A) = \log\lambda I + \log(A)$, this implies 
\[
    x^T \log(A) x -  \|x\|_2^2\cdot e_1^T \log(T_m) e_1 = x^T \log(B) x -  \|x\|_2^2 \cdot e_1^T \log(T_m^{B}) e_1.
\]
    Note that the smallest/largest eigenvalues of $B$ are given by $1/2$ and $\kappa(A)/2$, respectively. Applying Corollary~\ref{cor:convgeneral} to $B$ with\footnote{In fact, it is possible to choose 
    $\rho = \frac{\sqrt{\kappa(A)+\varepsilon}+1}{\sqrt{\kappa(A)+\varepsilon}-1}$ for arbitrary $\varepsilon > 0$.} 
    $\rho := \frac{\sqrt{\kappa(A)+1}+1}{\sqrt{\kappa(A)
    +1}-1}$ thus
    gives
    \[
    | x^T \log(A) x -  \|x\|_2^2\cdot e_1^T \log(T_m) e_1 | \le \| x \|_2^2 \cdot \frac{4 M_{\rho}}{1-\rho^{-1}} \rho^{-2m}.
    \]
    The constant $M_{\rho}$ is the maximum absolute value of the logarithm on the ellipse with foci $1/2$ and $\kappa(A)/2$ that intersects the real axis at
    $\alpha := \frac{1}{2\kappa(A)}$ and $\beta := \frac{\kappa(A)^2 + \kappa(A) - 1}{2\kappa(A)}$. By Corollary~\ref{cor:logellipse}, $M_\rho = |\log(\alpha)| = \log(2 \kappa(A))$, where we used $\alpha \le1/\beta \le 1$. Noting that
    \[
    \frac{4 M_{\rho}}{1-\rho^{-1}} = 
    2 (\sqrt{\kappa(A)+1}+1) \log(2\kappa(A))\, \new{= c_A}
    \]
    concludes the proof.
\end{proof}

 \section{Combined bounds for determinant estimation}

    Combining randomized trace estimation with the Lanczos method, we obtain the following (stochastic) estimate for $\log(\det(A))$:
    \begin{equation*}
        \estGR_{N,m} := \sum_{i=1}^N \| X^{(i)} \|_2^2 \cdot e_1^T \log(T_m^{(i)}) e_1,
    \end{equation*}
    where $X^{(1)}, \ldots, X^{(N)}$ are independent Gaussian or Rademacher random vectors and $T_m^{(i)}$ is the tridiagonal matrix obtained from the Lanczos method with starting vector $X^{(i)} / \| X^{(i)} \|_2$.  By combining the results obtained so far, we now derive new bounds on the number of samples and number of Lanczos steps needed to ensure an approximation error of at most $\varepsilon$ (with high probability).

    \subsection{Standard Gaussian random vectors}
    
     \begin{theorem}\label{thm:finalG}
        Suppose that the following holds for $N$ (number of Gaussian probe vectors) and $m$ (number of Lanczos steps per probe vector):
        \begin{itemize}
         \item[(i)] $N \ge 16 \varepsilon^{-2} (\rho_{\log} \| \log(A) \|_2^2 + \varepsilon \|\log(A) \|_2) \log\frac{4}{\delta}$, where $\rho_{\log}$ denotes the stable rank of $\log(A)$;
         \item[(ii)] $m \ge \frac{\sqrt{\kappa(A)+1}}{4} \log\big( \new{4\varepsilon^{-1} n^2(\sqrt{\kappa(A)+1}+1) \log (2\kappa(A))} \big)$.
        \end{itemize}
        If, additionally, \new{$n\ge 2$ and $N \le \frac{\delta}{2}  \exp \big( \frac{n^2}{16} \big)$} then
        $
            \mathbb{P}(|\estG_{N,m} - \log\det(A)| \ge \varepsilon) \le \delta.
        $
    \end{theorem}
    
    \begin{proof}

For a Gaussian vector $X$, the squared norm $\|X\|_2^2$ is a Chi-squared random variable with $n$ degrees of freedom. Therefore, by~\cite[Lemma 1]{Laurent2000} we have
        \begin{equation*}
            \mathbb{P}(\| X \|_2^2 \ge n + 2\sqrt{nt} + 2t) \le \exp(-t)
        \end{equation*} for every $t > 0$. For $t = \log\frac{2N}{\delta}$, the \new{additional assumptions} of the theorem imply
        \begin{equation*}
            n + 2\sqrt{nt} + 2t \le n + 2\sqrt{n} \cdot \frac{n}{4} + 2\cdot\frac{n^2}{16} < n^2,
        \end{equation*}
        and therefore $\mathbb{P}(\|X\|_2^2 \ge n^2) \le \frac{\delta}{2N}$.         By the union bound, it holds that
        \begin{equation}\label{eq:controlnorm}
            \mathbb{P}\left ( \text{exists } i \in \{1,\ldots,N\} \text{ s.t. }\| X^{(i)} \|_2^2 \ge n^2 \right ) \le \frac{\delta}{2}.
        \end{equation}
        Corollary~\ref{cor:numLanczosSteps}, together with condition (ii) and~\eqref{eq:controlnorm}
        imply that
        $
            | \estG_{N,m} - \traceG_N(\log(A)) | \le \frac{\varepsilon}{2}
        $
        holds with probability at least $1 - \delta/ 2$, where we also used that $ \log\left ( \frac{ \sqrt{\kappa(A)+1}+ 1} {\sqrt{\kappa(A)+1} -1}\right ) \ge \frac{2}{\sqrt{\kappa(A)+1}}$.
        
    Applying Theorem~\ref{thm:gauss} to the matrix $\log(A)$, for which \new{$\| \log(A) \|_F^2 = \rho_{\log}\| \log(A)\|_2^2$}, we find that
    $
        |\traceG_N(\log(A)) - \log\det(A)| \le \frac{\varepsilon}{2}
    $
    holds with probability at least $1 - \delta / 2$. The proof is concluded by applying the triangle inequality. 
    \end{proof}

    \subsection{Rademacher random vectors}
    
    \begin{theorem}\label{thm:finalR}
             Suppose that the following holds for $N$ (number of Rademacher probe vectors) and $m$ (number of Lanczos steps per probe vector): 
        \begin{itemize}
        \item[(i)] $N \ge 32 \varepsilon^{-2} \left ( \rho_{\mathrm{logd}} \| \log(A) - \D_{\log(A)} \|_2^2 + \frac{\varepsilon}{2} \| \log(A) - \D_{\log(A)} \|_2 \right ) \log \frac{2}{\delta}$, where $\rho_{\mathrm{logd}}$ denotes the stable rank of 
        $\log(A) - \D_{\log(A)}$ and $\D_{\log(A)}$ is the diagonal matrix containing the diagonal entries of $\log(A)$;
         \item[(ii)] $m \ge \frac{\sqrt{\kappa(A)+1}}{4} \log \big( \new{4 \varepsilon^{-1} n(\sqrt{\kappa(A) + 1} + 1)\log(2\kappa(A)) }\big)$.
        \end{itemize}
        Then $\mathbb{P}(|\estR_{N,m} - \log\det(A)| \ge \varepsilon) \le \delta$.
    \end{theorem}
    
    \begin{proof}
    Using Corollary~\ref{cor:convgeneral} and the fact that Rademacher random vectors have norm $\sqrt{n}$, the bound
$
  \big| \estR_{N,m} - \traceR_N(\log(A)) \big| \le \frac{\varepsilon}{2}
$
holds if
\[
m \ge \frac{1}{2} \log\big( \new{4 \varepsilon^{-1} n(\sqrt{\kappa(A)+1}+1) \log(2\kappa(A)) } \big) \big / \log\left ( \frac{ \sqrt{\kappa(A)+1}+ 1} {\sqrt{\kappa(A)+1} -1}\right ).
\]
Because of $\log\Big( \frac{ \sqrt{\kappa(A)+1}+ 1} {\sqrt{\kappa(A)+1} -1}\Big) \ge \frac{2}{\sqrt{\kappa(A)+1}}$, condition (ii) ensures that this inequality is satisfied.

Applying Corollary~\ref{cor:rademacher} to $\log(A)$ and with $\varepsilon$ replaced by $\varepsilon/2$, immediately shows 
    \begin{equation}\label{eq:wanted}
        |\traceR_N(\log(A)) - \log\det(A)| \le \frac{\varepsilon}{2}
    \end{equation}
with probability at least $1 - \delta$ if condition (i) is satisfied.
The proof is concluded by applying the triangle inequality. 
    \end{proof}
    
    \begin{paragraph}{Comparison with existing result.}
    To compare Theorem~\ref{thm:finalR} with an existing result from~\cite{Ubaru2017}, it is helpful to first derive a simpler (but usually stronger) condition on $N$.
    \begin{lemma} \label{lemma:simplerN} The statement of Theorem~\ref{thm:finalR} holds with condition (i) replaced by $$N \ge 8 \varepsilon^{-2} \big( n \log^2 \kappa(A) + 2\varepsilon \log \kappa(A) \big) \log \frac{2}{\delta}.$$
    \end{lemma}
    \begin{proof}
    We set $B := \lambda A$ with 
    $\lambda := 1/\sqrt{\lambda_{\min}(A) \lambda_{\max}(A)}$ and note that  $$\traceR_N(\log(A)) - \log\det(A) = \traceR_N(\log(\lambda A)) - \log\det(\lambda A).$$ 
    Using $\lambda_{\max}(B) = \sqrt{\kappa(A)}$, $\lambda_{\min}(B) = 1/\sqrt{\kappa(A)}$, and $\kappa(B) = \kappa(A)$, we obtain
    \begin{align*}
        \| \log(B) - \D_{\log(B)} \|_2 & \le 2 \| \log(B) \|_2 = \log \kappa(A);\\
        \| \log(B) - \D_{\log(\new{B})} \|_F^2 & \le \| \log(B) \|_F^2 =  \rho(\log(B)) \frac{ \log^2\kappa(A)}{4} \le \frac{n}{4} \log^2\kappa(A).
    \end{align*}
    An application of Corollary~\ref{cor:rademacher} to $\log(B)$ therefore yields~\eqref{eq:wanted} with probability at least $1-\delta$ for $N \ge 8 \varepsilon^{-2} \big( n \log^2 \kappa(A) + 2 \varepsilon \log \kappa(A) \big) \log\frac{2}{\delta}$.
\end{proof}
    
    Correcting for the two minor erratas explained above, the result from~\cite[Corollary 4.5]{Ubaru2017} states that
    $\mathbb{P}(|\estR_{N,m} - \trace(\log(A))| \ge \varepsilon) \le \delta$
    holds if 
    \begin{equation} \label{eq:condN}
    N \ge 24 \varepsilon^{-2} n^2 \left ( \log(1 + \kappa(A)) \right )^2 \log \frac{2}{\delta}
    \end{equation}
    and
    \begin{equation} \label{eq:condm}
    m \ge \frac{\sqrt{3\kappa(A)}}{4} \log \big( 20 \varepsilon^{-1}  n \big(  \sqrt{2\kappa(A)+1} + 1\big)\log(2\kappa(A)+2) \big).
    \end{equation}
    \new{Compared to~\eqref{eq:condN}, Lemma~\ref{lemma:simplerN} reduces the explicit dependence on the matrix size from $n^2$ to $n$. The dependence of the bounds on $\kappa(A)$ is comparable.}
    \new{Let us stress that even a dependence on $n$ does not compare favorably to simply computing the diagonal elements, but the bound from condition (i) of Theorem~\ref{thm:finalR} can often be expected to be significantly better than the simplified bound of Lemma~\ref{lemma:simplerN}. Below we describe a situation in which the former only depends logarithmically on $n$.}
Condition (ii) of Theorem~\ref{thm:finalR} improves~\eqref{eq:condm} clearly but less drastically, roughly by a factor $\sqrt{3}$.
    
    \end{paragraph}
    
    \begin{paragraph}{\new{Implications of low stable rank.}}
    Let us consider a family of matrices $\{ A_n \}$ of increasing dimension, a fixed failure probability $\delta$, and a fixed accuracy $\varepsilon$; the number of probe vectors required to get $\mathbb{P} (|\traceGR_N(\log(A_n)) - \trace(\log(A_n))| \ge \varepsilon) \le \delta$ is proportional to $O(\rho_{n} \|\log(A_n)\|_2^2)$, where $\rho_n$ is the stable rank of $\log(A_n)$. In certain applications, including regularized kernel matrices (see, e.g.,~\cite{Chen2013,Gardner2018}), the stable rank grows slowly when the matrix size increases. For such situations, our bounds lead to favorable implications. To illustrate this, let us consider matrices $A_n := I + B_n$, where the eigenvalues satisfy $\lambda_i(B_n) \le nC\alpha^i$ for some constants $C>0$ and $0 < \alpha < 1$, for all $i\le n$, such as in the discretization of a radial basis function kernel on a fixed domain~\cite{Gardner2018}. In this case, $\rho_n = O(\log n)$. As a second example, if $B_n$ comes from a discretization of a Mat\'ern kernel on a regular grid in a fixed domain, its eigenvalues satisfy $\lambda_i(B_n) \le nCi^{-\beta}$ for some constants $C>0$ and $\beta>1$, for all $i \le n$~\cite{Chen2013}; the stable rank of $\log(A_n) = \log(I + B_n)$ is bounded by $\rho_{n} = O(n^{1/\beta})$.
    To apply Theorems~\ref{thm:finalG} and \ref{thm:finalR} one also needs to take into account that, for both our examples, $\|\log(A_n)\|_2$ and $\kappa(A_n)$ grow proportionally to $\log(n)$ and $n$, respectively. Finally, note that in practice one would consider $A_n = \sigma I + B_n$ with the regularization parameter $\sigma$ chosen adaptively; see, e.g.,~\cite{Caponnetto2007}.
    \end{paragraph}
    
    \new{
    \section{Numerical experiments}
    
    In this section, we report on a number of numerical experiments illustrating the bounds obtained in this work.
     All numerical experiments have been performed in {\sc Matlab}, version 9.9 (R2020b).
    
    \begin{example} \label{example:triangles}
         To compare the estimates from Theorem~\ref{thm:gauss} and Corollary~\ref{cor:rademacher} with the convergence of randomized trace estimation using Gaussian and Rademacher vectors, we use an example from~\cite{Avron2010, meyer2021hutch++}. The number of triangles in an undirected graph is equal to $\frac{1}{6} \trace(A^3)$ where $A$ is the (usually indefinite) adjacency matrix. Note that the quadratic forms $X^T A^3 X$ can be evaluated exactly using two matrix-vector multiplications.
    The results for an arXiv collaboration network with $n = 5\,242$ nodes and $48\,260$ triangles.\footnote{See \url{https://snap.stanford.edu/data/ca-GrQc.html}.}  
    
    We estimate $\trace(A^3)$ using $N = 2, 2^2, 2^3, \ldots, 2^{11}$ samples. For each value of $N$ we performed $1\,000$ experiments and discarded the 5\% worst approximations in order to estimate an error bound that holds with probability $95\%$.  The obtained results are represented by the shaded regions in Figure~\ref{fig:tri1} and match the obtained bounds fairly well, especially for Gaussian vectors.
    
    Figure~\ref{fig:tri2} shows the empirical failure probability $\mathbb{P}(|\traceGR_N(A^3) - \trace(A^3)| \ge \varepsilon)$ with $\varepsilon = \frac{1}{10} \trace(A^3)$ using $1\,000$ experiments for $N = 2, 2^2, 2^3, \ldots, 2^{11}$ (blue and red lines). The vertical purple and yellow lines are the estimated number of samples needed to achieve failure probability $\delta = 0.05$ from Theorem~\ref{thm:gauss} and Corollary~\ref{cor:rademacher}, respectively. 
\end{example}

    \begin{figure}
        \centering
        \includegraphics{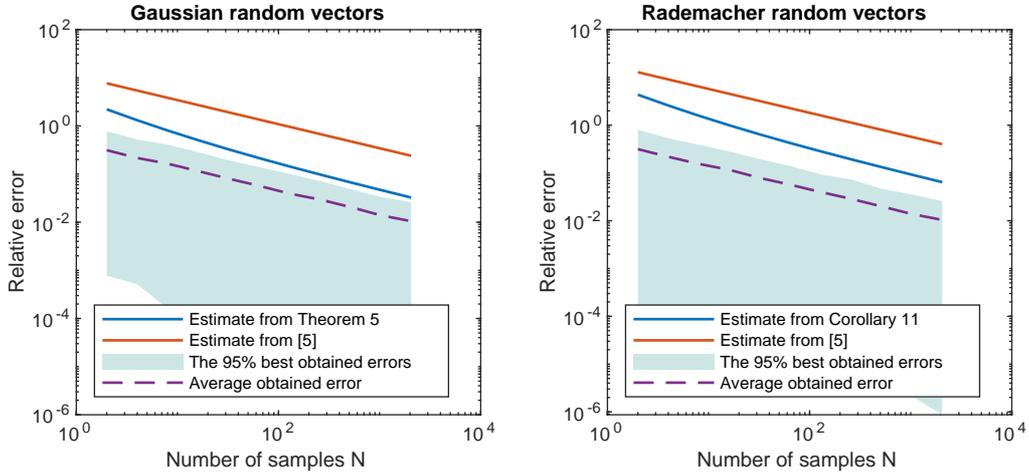}
        \caption{\new{Estimation of $\trace(A^3)$ with Gaussian and Rademacher vectors for the matrix from Example~\ref{example:triangles}. 
        Error bounds from Theorem~\ref{thm:gauss}, Corollary~\ref{cor:rademacher}, and~\cite{Avron2010} for failure probability $\delta = 0.05$ compared with the observed error.      
        }}
        \label{fig:tri1}
    \end{figure}
    
    \begin{figure}
        \centering
        \includegraphics{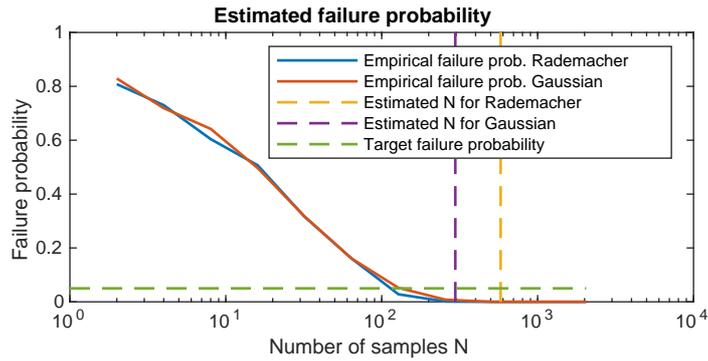}
        \caption{\new{Number of samples needed to attain error $\varepsilon = \frac{1}{10} \trace(A^3)$ with failure probability $5\%$ for Example~\ref{example:triangles}. Empirical failure probability vs. bounds from Theorem~\ref{thm:gauss} and Corollary~\ref{cor:rademacher}.}}
        \label{fig:tri2}
    \end{figure}
    
    \begin{table}
 \centering\new{
 \begin{tabular}{| c | c | c | c | c | c | c |}
  \hline
  Name & Size & Ref. & $\log\det(A)$ & $\kappa(A)$ & $\|\log(A)\|_F$ & $\|\log(A) - \D_{\log(A)}\|_F$ \\
  \hline
  \texttt{thermomec\_TC} & $102158$ & \cite{UFL} & $-5.47 \cdot 10^5$ & $67.2$ & $1.72 \cdot 10^3$ & $122.8$ \\
  \hline
  \texttt{lowrank} & 20000 & \cite{LiZhu2020} & $89.4$ & $1560$ & $17.04$ & $16.99$ \\
  \hline 
  \texttt{precip} & $6400$ & \cite{meyer2021hutch++} & $-2.56 \cdot 10^4$ & $6738$ & $357$ & $157$ \\
  \hline
 \end{tabular}
 \caption{Summary of the matrices used for log determinant experiments.}
 \label{table:summary}}
\end{table}

\begin{example}
To compare the results of Theorems~\ref{thm:finalG} and~\ref{thm:finalR} with the number of sample vectors $N$ and Lanczos steps $m$ (per sample) required to reach a fixed accuracy, we consider the matrices listed in Table~\ref{table:summary}. The matrix \texttt{thermomec\_TC} is contained in the University of Florida sparse matrix collection~\cite{UFL} and has been considered, for instance, in~\cite{Boutsidis2017,Fitzsimons2017,Ubaru2017}. The matrix \texttt{lowrank} is defined in~\cite{Saibaba2017,LiZhu2020} as
    \begin{equation*}
A = \sum_{j=1}^{40} \frac{10}{j^2} x_j x_j^T + \sum_{j=41}^{300} \frac{1}{j^2} x_j x_j^T,
\end{equation*}
where each $x_j$ is a sparse vector of length $20\,000$ with approximately $2.5\%$ uniformly distributed nonzero entries, generated with the {\sc Matlab} command \texttt{sprand}.
The matrix \texttt{precip} is a two-dimensional Gaussian kernel matrix with length parameter $\gamma = 64$ and regularization parameter $\lambda = 0.008$ taken from~\cite{meyer2021hutch++}, involving precipitation data from Slovakia~\cite{Neteler2013}.  As the matrices \texttt{thermomec\_TC} and \texttt{lowrank} are too large for $\log(A)$ to be computed explicitly, the quantities $\| \log(A)\|_F$ and $\|\log(A) - \D_{\log(A)}\|_F$ are approximated by randomized trace estimation combined with the Lanczos method to estimate the diagonal elements of $\log(A)$.

For quadratic forms involving the logarithm, there is a relatively inexpensive way to obtain an upper bound on the error of the Lanczos method. As discussed in~\cite{Bai1996}, Gauss quadrature always yields an upper bound for $x^T \log(A)x$, while Gauss-Lobatto quadrature always yields a lower bound. We fix $\delta = 0.1$ and for several values of $\varepsilon$ we investigate how many samples and Lanczos iterations are needed in practice. When approximating quadratic forms while aiming at accuracy $\varepsilon$, we stop the Lanczos method when the difference between upper and lower bound is less than $\varepsilon/2$. 
Starting from $N = 1$, we compute the empirical failure probability  $\mathbb{P}(|\mathrm{est}_{N,m} - \log \det(A)| \ge \varepsilon)$; if this probability is larger than $\delta$, we double the number of samples $N$ and repeat.

The results for the three matrices from Table~\ref{table:summary} are reported in Figures~\ref{fig:thermo},~\ref{fig:lowrank}, and~\ref{fig:precip}. The left plots show, for the considered values of $\varepsilon$ (which have been normalized by dividing them by the true $|\log \det(A)|$), the number of samples required to attain $90\%$ success probability over $30$ runs of the algorithm, versus the number of samples given by Theorems ~\ref{thm:finalG} and~\ref{thm:finalR}. The plots on the right show, for the same (normalized) values of $\varepsilon$, the average number of Lanczos steps required to reach accuracy $\varepsilon/2$
versus the number of Lanczos steps predicted by Theorems~\ref{thm:finalG} and~\ref{thm:finalR}.

For \texttt{thermomec\_TC}, the diagonal of $\log(A)$ is large relative to the rest of the matrix: $\|\log(A)- \D_{\log(A)}\|_F/\| \log(A)\|_F \approx 0.07$. Therefore, our bounds predict that Rademacher vectors perform much better than Gaussian vectors; this is indeed confirmed by Figure~\ref{fig:thermo}. The matrix $A$ is well conditioned and, hence, the bounds correctly predict that the Lanczos method only needs relatively few iterations to attain good accuracy. 

For \texttt{lowrank}, Figure~\ref{fig:lowrank} shows that Rademacher and Gaussian vectors perform similarly. Although the condition number of $A$ is $\kappa(A) \approx 1560$, the eigenvalues have a strong decay, and hence its adaptivity lets the Lanczos method perform much better than predicted by our bounds, see, e.g.,~\cite{Guttel2020} for a discussion. 

For \texttt{precip}, the ratio $\|\log(A)- \D_{\log(A)}\|_F/\| \log(A)\|_F \approx 0.44$ is reflected in Figure~\ref{fig:precip}, showing that Rademacher vectors attain somewhat better accuracy. The condition number of $A$ is high and there is no strong decay or gaps in the singular values; a relatively large number of Lanczos steps is necessary to obtain the desired accuracy when approximating the quadratic forms.
\end{example}

\begin{figure}
    \centering
    \includegraphics[scale=.95]{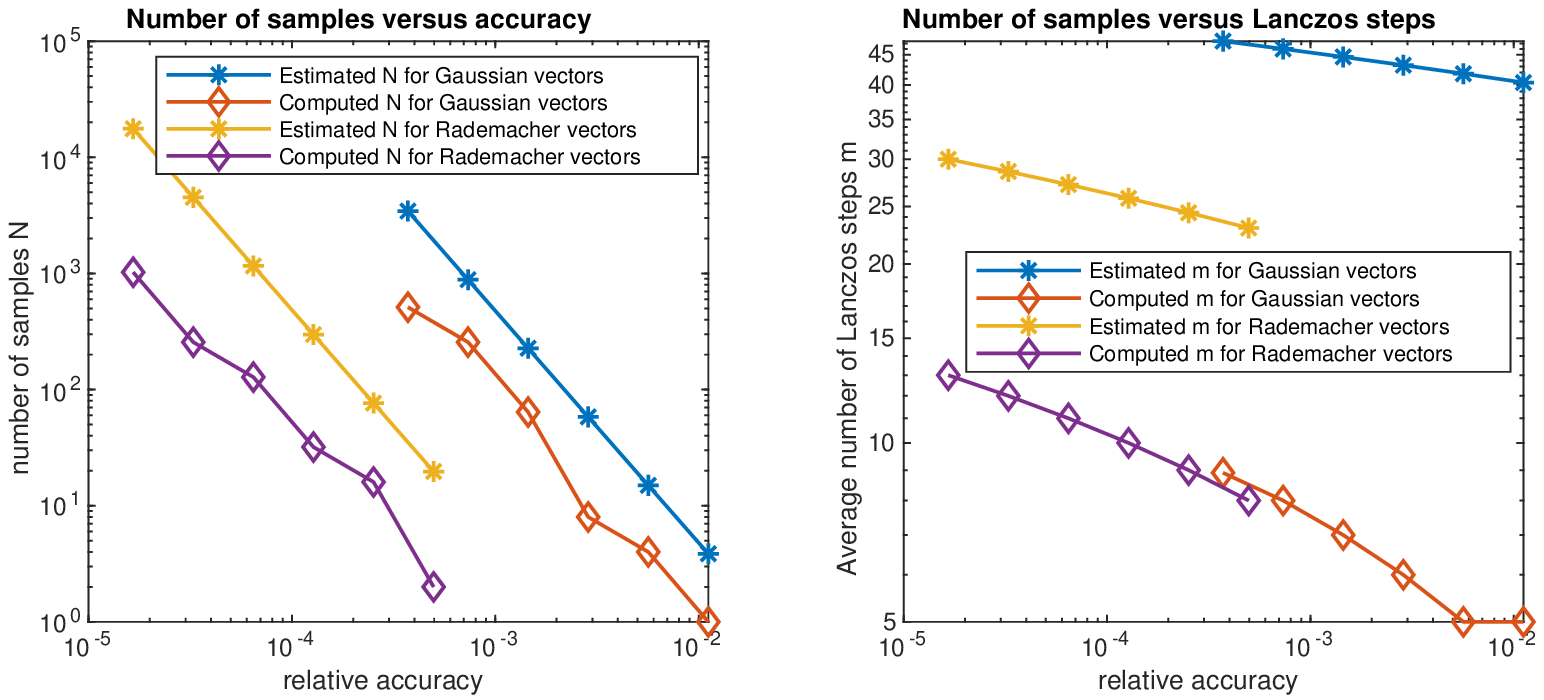}
    \caption{\new{Results for matrix \texttt{thermomec\_TC} from~\cite{UFL}. }}
    \label{fig:thermo}
\end{figure}

\begin{figure}
    \centering
    \includegraphics[scale=.95]{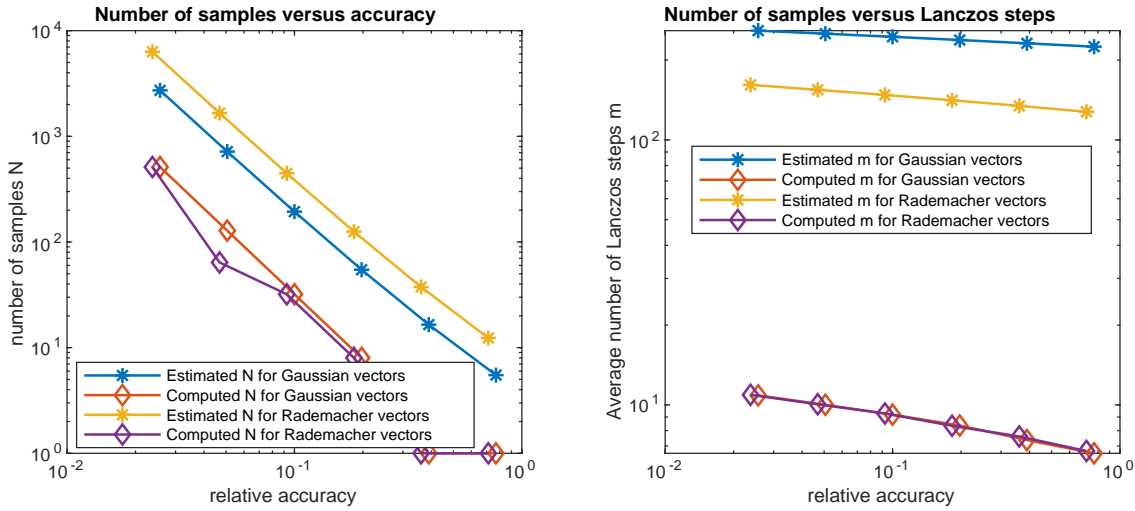}
    \caption{\new{Results for matrix \texttt{lowrank} from~\cite{LiZhu2020}.}}
    \label{fig:lowrank}
\end{figure}

\begin{figure}
    \centering
    \includegraphics[scale=.95]{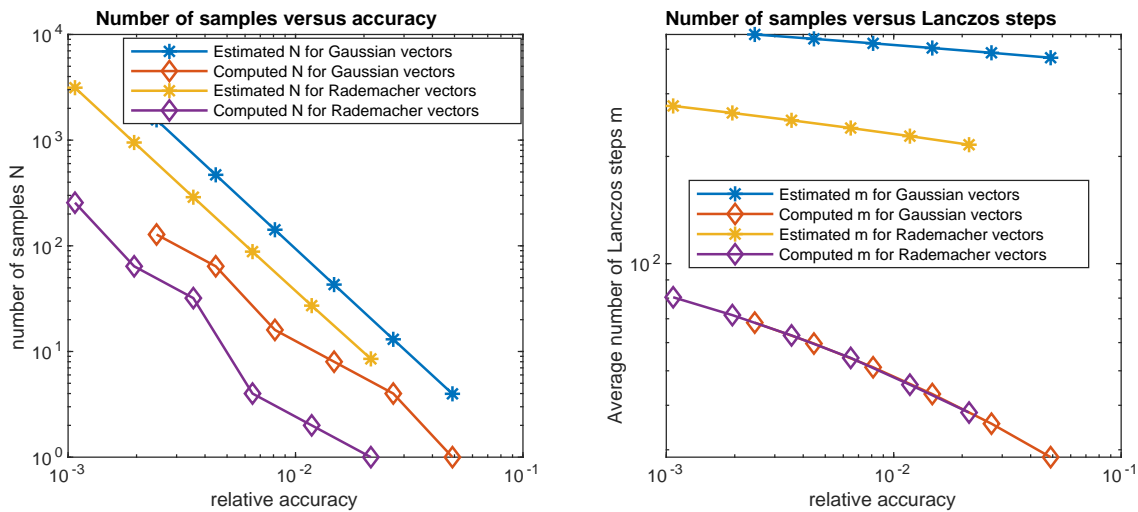}
    \caption{\new{Results for matrix \texttt{precip} from~\cite{meyer2021hutch++}.}}
    \label{fig:precip}
\end{figure}}

\begin{paragraph}{Acknowledgments.}
We thank Rados{\l}aw Adamczak, Rasmus Kyng, and Shashanka Ubaru for helpful discussions on topics related to this work. \new{We also thank the referees for providing valuable feedback.}
\end{paragraph}

	\bibliographystyle{abbrv}
	\bibliography{Bib} 
	
	\appendix

	\section{Auxiliary results}

\subsection{Herbst argument and logarithmic Sobolev inequalities}

This section contains auxiliary results used in the proof of Theorem~\ref{thm:concRademacher}. We recall that the entropy of a random variable $Z$ is defined as
\begin{equation*}
\ent(Z) := \mathbb{E}[Z \log Z] - \mathbb{E}[Z] \log \mathbb{E}[Z],
\end{equation*}
provided that all the involved expected values exist.

The Herbst argument (see, e.g.,~\cite[page 11]{Boucheron2013},~\cite[pages 239--240]{FoucartRauhut2013}, and~\cite[Section 3.1.2]{Wainwright2019}) turns a bound on the entropy of a random variable into a bound on the moment generating function. By Chernoff's bound, the latter implies a bound on the tail of the random variable. Specifically, we use the following (modified) Herbst argument.
\begin{lemma}\label{lemma:herbst}
Let $Z$ be a random variable and $g:[0, a) \to\R$ such that
\begin{equation}\label{eq:ineqH}
    \mathbb{H}(\exp(\lambda Z)) \le \lambda^2 g(\lambda) \mathbb{E}[\exp(\lambda Z)].
\end{equation}
Then for all $\lambda \in [0, a)$ it holds
\begin{equation*}
    \log\mathbb{E}[\exp(\lambda Z)] \le \lambda \mathbb{E}[Z] + \lambda \int_0^{\lambda} g(\xi) \mathrm{d}\xi.
\end{equation*}
\end{lemma}
\begin{proof}
For $\psi(\lambda) := \log \mathbb{E}[\exp(\lambda Z)]$, it holds that $\psi'(\lambda) = \mathbb{E}[Z\exp(\lambda Z)] / \mathbb{E}[\exp(\lambda Z)]$. Recalling the definition of entropy, this allows us to rewrite~\eqref{eq:ineqH} as
\begin{equation*}
    \lambda \psi'(\lambda) \exp(\psi(\lambda)) - \psi(\lambda) \exp(\psi(\lambda)) \le \lambda^2 g(\lambda)\exp(\psi(\lambda)),
\end{equation*}
which is equivalent to
\begin{equation*}
    \frac{\mathrm{d}}{\mathrm{d}\lambda} \left ( \frac{\psi(\lambda)}{\lambda}\right ) \le f(\lambda).
\end{equation*}
Integration on the interval $[0, \lambda]$ gives
\begin{equation*}
    \frac{\psi(\lambda)}{\lambda} - \lim_{\lambda\to 0^+} \frac{\psi(\lambda)}{\lambda} \le \int_0^{\lambda} f(\xi)\mathrm{d}\xi.
\end{equation*}
We conclude by noting that $\lim_{\lambda\to 0^+} \frac{\psi}{\lambda} = \mathbb{E}[Z]$.
\end{proof}

For deriving bounds on the entropy, we need the following two variations of  Gross' logarithmic Sobolev inequality.

\begin{theorem}\label{thm:logsobolev}
Let $f:\{-1,1\}^n \to \R$ and let $X$ be a Rademacher vector with components $X_1, \ldots, X_n$. Define $f(\bar{X}^{(i)}) := f(X_1, \ldots, X_{i-1}, -X_i, X_{i+1}, \ldots, X_n)$ for $i=1, \ldots, n$. Then for all $\lambda > 0$ we have
\begin{equation}\label{eq:grossplus}
    \ent(\exp(\lambda f(X))) \le \frac{\lambda^2}{4} \mathbb{E} \left [ \exp(\lambda f(X)) \sum_{i=1}^n \left ( f(X) - f(\bar{X}^{(i)}) \right )_+^2 \right ]
\end{equation}
and
\begin{equation}\label{eq:grossnormal}
    \ent(\exp(\lambda f(X))) \le \frac{\lambda^2}{8} \mathbb{E} \left [ \exp(\lambda f(X)) \sum_{i=1}^n \left ( f(X) - f(\bar{X}^{(i)}) \right )^2 \right ].
\end{equation}
\end{theorem}

\begin{proof}
The inequality~\eqref{eq:grossplus} is a standard result  and can be found, e.g.,  in~\cite[page 122]{Boucheron2013}. The inequality~\eqref{eq:grossnormal} is a variation of the same argument; see also~\cite[Exercise 5.5]{Boucheron2013} for a related (but not identical) result. The inequality~\eqref{eq:grossnormal} can, in fact, be found in a Master's thesis~\cite[Theorem 5]{Adamczak2003}. For convenience of the reader, we provide a proof of~\eqref{eq:grossnormal} based on the textbook~\cite{Boucheron2013}.

In~\cite[page 122]{Boucheron2013} it is proven that
\begin{equation}\label{eq:bouch1}
    \ent(\exp(\lambda f(X))) \le \frac{1}{2} \mathbb{E} \left [ \sum_{i=1}^n \left ( \exp(\lambda f(X)/2) - \exp(\lambda f(\bar{X}^{(i)})/2) \right )^2 \right ].
\end{equation}
For $a\ge b$ we have
\begin{align*}
    \exp\left (\frac{a}{2}\right) - \exp\left( \frac{b}{2} \right ) & = \int_{b/2}^{a/2} \exp(t)\mathrm{d}t \le \frac{a-b}{2} \cdot \frac{\exp\left (\frac{a}{2}\right) + \exp\left( \frac{b}{2} \right )}{2}\\
    & \le \frac{a-b}{2} \sqrt{\frac{\exp(a)+\exp(b)}{2}},
\end{align*}
where the first inequality follows from the concavity of the exponential and the Hermite-Hadamard inequality. Therefore, for all $a, b\in \R$ we have
\begin{equation}\label{eq:ineqab}
(\exp(a/2) - \exp(b/2))^2 \le \frac{1}{8} (a-b)^2(\exp(a) + \exp(b)).
\end{equation}
Applying~\eqref{eq:ineqab} to each summand in Equation~\eqref{eq:bouch1} one obtains
\begin{align*}
    \ent(\exp(&\lambda f(X)))  \le \frac{\lambda^2}{16} \sum_{i=1}^n \mathbb{E} \left [ (f(X) - f(\bar{X}^{(i)}))^2  \left ( \exp(\lambda f(X)) + \exp(\lambda f(\bar{X}^{(i)})) \right )\right ]\\ 
    & =  \frac{\lambda^2}{16} \sum_{i=1}^n \mathbb{E} \left [ (f(X) - f(\bar{X}^{(i)}))^2  \exp(\lambda f(X))\right ] \\ & + \frac{\lambda^2}{16} \sum_{i=1}^n \mathbb{E} \left [ (f(X) - f(\bar{X}^{(i)}))^2  \exp(\lambda f(\bar{X}^{(i)}))\right ]\\
    & = \frac{\lambda^2}{8} \mathbb{E} \left [ \exp(\lambda f(X)) \sum_{i=1}^n \left ( f(X) - f(\bar{X}^{(i)}) \right )^2 \right ],
\end{align*}
where the last equality follows from the fact that $f(X)$ and $f(\bar{X}^{(i)})$ have the same distribution and changing the sign of the $i$th entry of $\bar{X}^{(i)}$ gives $X$ again.
\end{proof}

\subsection{Bounds on the complex logarithm}

The following two elementary results on the complex logarithm are needed in the convergence proofs of the Lanczos method in Section~\ref{sec:Lanczos}.
\begin{lemma}\label{lemma:circle}
Consider a circle in the complex plane with center $a \in \R^+$, $a>1$ and radius $b$ such that $b^2 = a^2 - 1$. Then the maximum absolute value of the logarithm on this circle is attained on the real axis.
\end{lemma}
\begin{proof}
By symmetry, we can restrict ourselves to the upper half of the circle. For fixed $\theta \in [0, \arctan(b)]$ the line $r e^{  \mathrm{i} \theta}$ for $r> 0$ intersects the circumference when $(r \cos\theta - a)^2 + r^2 \sin^2\theta = b^2$. Clearly, this equality holds for $r_{\pm} = a \cos\theta \pm \sqrt{a^2\cos^2\theta - 1}$. Note that these points parametrize the entire upper semi-circle and we have $r_- = \frac{1}{r_+}$, so 
\begin{equation*}
    |\log(r_- e^{ \mathrm{i} \theta})| = |\log(r_-) +  \mathrm{i} \theta| = |-\log(r_+) +  \mathrm{i} \theta| = |\log(r_+) +  \mathrm{i} \theta| = |\log(r_+ e^{  \mathrm{i} \theta})|.
\end{equation*}
Therefore, to prove the lemma it is sufficient to show that the function $g: [0, \arctan(b)] \to \R$ given by $g(\theta) := |f(\theta)|^2$, where $f(\theta) = \log(r_+ e^{ \mathrm{i} \theta})$, attains its maximum for $\theta = 0$. We will establish this fact by showing that $g$ decreases monotonically. We have
\begin{align*}
    f(\theta) & = \log\left ( \big( a \cos\theta + \sqrt{a^2\cos^2\theta-1}\big) e^{ \mathrm{i} \theta}\right ) = \arcsinh\big( \sqrt{a^2\cos^2\theta-1}\big) + \mathrm{i} \theta; \\
    f'(\theta) & = \frac{-a^2\cos\theta\sin\theta}{a\cos\theta \sqrt{a^2\cos^2\theta - 1}} +  \mathrm{i} = - \frac{a\sin\theta}{\sqrt{a^2\cos^2\theta - 1}} +  \mathrm{i},
\end{align*}
and therefore
\begin{equation*}
    g'(\theta) = 2 \mathrm{Re} \big(f'(t) \cdot \overline{f(t)} \big) = -2\frac{a\sin\theta\cdot \arcsinh\big( \sqrt{a^2\cos^2\theta - 1} \big)}{\sqrt{a^2\cos^2\theta - 1}} + 2\theta.
\end{equation*}
For $\theta \in (0, \arctan(b))$ we have that
\begin{equation*}
    g'(\theta) \le 0 \Leftrightarrow \frac{\theta}{a\sin\theta} \le \frac{\arcsinh\left ( \sqrt{a^2\cos^2\theta - 1}\right )}{\sqrt{a^2\cos^2\theta - 1}}.
\end{equation*}
Using the facts that $x\mapsto \frac{x}{\sin(x)}$ is increasing for $x \in [0, \pi]$, $\arctan(x) < \arcsinh(x)$ for $x > 0$, and $x\mapsto \frac{\arcsinh(x)}{x}$ is decreasing for $x>0$, one obtains
\begin{equation*}
    \frac{\theta}{a\sin\theta} \le \frac{\arctan(b)}{b} < \frac{\arcsinh(b)}{b} \le \frac{\arcsinh\left ( \sqrt{a^2\cos^2\theta - 1}\right )}{\sqrt{a^2\cos^2\theta - 1}},
\end{equation*}
for every $\theta \in (0, \arctan(b))$. In particular, this shows 
 $g'(\theta) < 0$ for  $\theta \in (0, \arctan(b))$ and hence $g$ is decreasing.
\end{proof}

\begin{corollary}\label{cor:logellipse}
Consider an ellipse $\mathcal{E}$ in the open right-half complex plane, with foci on the real axis. Then the maximum absolute value of the logarithm on this ellipse is attained on the real axis. 
\end{corollary}
\begin{proof}
Let $0 < \alpha < \beta$ be the two intersections of the ellipse with the real axis. If $|\log \alpha| \ge |\log \beta|$ then $\mathcal{E}$ is contained in the circle $\mathcal{C}_1$ of center $a := \frac{1}{2} \left (  \frac{1}{\alpha} + \alpha \right )$ and radius $b := \frac{1}{2} \left (  \frac{1}{\alpha} - \alpha \right )= \sqrt{a^2 - 1}$, and $\mathcal{E}$ is tangent to $\mathcal{C}_1$ in $\alpha$; otherwise $\mathcal{E}$ is contained in the circle $\mathcal{C}_2$  of center $a := \frac{1}{2} \left ( \beta + \frac{1}{\beta} \right )$ and radius $b := \frac{1}{2} \left ( \beta - \frac{1}{\beta} \right ) = \sqrt{a^2-1}$, and $\mathcal{E}$ is tangent to $\mathcal{C}_2$ in $\beta$. 
In both cases, the result follows from Lemma~\ref{lemma:circle}.
\end{proof}
	
\end{document}